\documentclass[10pt]{amsart}

\usepackage{tikz}
\usepackage{amsmath}
\usepackage{amssymb}
\usepackage{mathrsfs}
\usepackage[all]{xy}
\usepackage[T1]{fontenc}
 \usepackage{textcomp}
\usepackage{graphicx}
\usepackage{orcidlink}

\newtheorem{theorem}{Theorem}[section]

\newtheorem{lemma}[theorem]{Lemma}

\theoremstyle{definition}
\newtheorem{definition}[theorem]{Definition}

\theoremstyle{remark}

\newcommand{\Lauchli}{L{\"{a}}uchli}

\newcommand{\re}{\upharpoonright}
\newcommand{\ra}{\rightarrow}

\newcommand{\cC}{\mathcal{C}}
\newcommand{\cD}{\mathcal{D}}

\DeclareMathOperator{\ran}{ran}
\DeclareMathOperator{\dom}{dom}

\newcommand{\om}{\omega}

\newcommand{\sse}{\subseteq}
\newcommand{\contains}{\supseteq}

\DeclareMathOperator{\Lev}{Lev}
\DeclareMathOperator{\AC}{AC}
\DeclareMathOperator{\Con}{Con}

\title[Big Ramsey degrees of Henson graphs in $\mathrm{ACA}_0$]{
The finite big Ramsey degrees of Henson graphs are provable in  $\mathrm{ACA}_0$}
\date{}

\author{Peter Cholak}
\address{University of Notre Dame}
\email{cholak@nd.edu}
\author{Natasha Dobrinen}
\address{University of Notre Dame}
\email{ndobrine@nd.edu}
\author{Henry Towsner}
\address{University of Pennsylvania}
\email{htowsner@math.upenn.edu}


\thanks{Cholak was partially supported by NSF-DMS-2502292 and Erwin Schr{\"o}dinger Institute. 
ORCID (Cholak):
\href{https://orcid.org/0000-0002-6547-5408}{\orcidlink{0000-0002-6547-5408}{0000-0002-6547-5408}}. Dobrinen
was partially supported by NSF-DMS-2300896 and Erwin Schr{\"o}dinger
Institute. ORCID (Dobrinen):
\href{https://orcid.org/0000-0001-6360-0001}{\orcidlink{0000-0001-6360-0001}{0000-0001-6360-0001}}. Towsner
was partially supported by NSF-DMS-2054379 and NSF-DMS-2452009 and Erwin Schr{\"o}dinger Institute.}

\begin{document}

\begin{abstract} 
Let $\mathbb{H}_{n+1}$ denote a computable copy of the $(n+1)$-clique free universal homogeneous Henson graph,  $G$ denote a finite subgraph of $\mathbb{H}_{n+1}$, and $k(G,n)$ denote the big Ramsey degree of $G$ in $\mathbb{H}_{n+1}$.
We prove  that for any computable coloring $\chi$ of the copies of $G$ in $\mathbb{H}_{n+1}$, there is a copy $\mathbb{H}'$ of $\mathbb{H}_{n+1}$ 
that is computable from $0^{(2\delta(G,n)-1)}$
in which $\chi$ takes no more than $k(G,n)$ colors, where 
 $\delta(G,n)$, see Definition~\ref{defn.deltaG},
 denotes the maximum number of levels of a diary for $G$ in $\mathbb{H}_{n+1}$ (this is a finite number).
It follows that 
 the statement,
``Henson graphs have finite big Ramsey degrees," is provable in ACA$_0'$.  
Combining this with a recent result of  Cholak, Dobrinen, and McCoy \cite{CDM} yields the equivalence of  
 the statement with    ACA$_0'$ over RCA$_0$.

\end{abstract}

\maketitle

\section{Introduction}

Suppose we have a finite $(n+1)$-clique free graph $G$ ($n\geq 2$) and let $\chi$ be finite coloring of the copies of $G$ in 
the universal homogeneous  $(n+1)$-clique free graph, denoted by $\mathbb{H}_{n+1}$ and called a Henson graph. 
The structure $\mathbb{H}_{n+1}$ has finite big Ramsey degrees \cite{MR4128725, MR4568266}, which means there is a $k$ (depending only on $n$ and $G$) such that for all $\chi$ there is a subcopy $\mathbb{H}'$ of $\mathbb{H}_{n+1}$ so that $\chi$ takes on at most $k$ colors on the copies of $G$ in $\mathbb{H}'$.
The least such $k$ is called the {\em big Ramsey degree} of $G$ in $\mathbb{H}_{n+1}$, denoted here as in \cite{CDM} by $k(G,n)$.
These exact numbers were characterized in \cite{Balko7} in terms of so-called diaries.

In this paper, we show that when  $\mathbb{H}_{n+1}$ and  $\chi$ are computable,  then $\mathbb{H}'$ is computable from $0^{(2 \delta(G,n)-1)}$.
Here,
 $\delta(G,n)$ denotes the maximum number of levels in  a diary for $G$ in $\mathbb{H}_{n+1}$, see Definition \ref{defn.deltaG}.
 $\delta(G,n)$ is computable by work in \cite{Balko7}, and is roughly bounded  by 
$2|G|(n-1)\sum_{1\le k\le n}{|G|\choose k}$.
Hence, for each fixed $G$ this proof can be carried out in $\mathsf{ACA}_0$.
This improves on a result  in \cite{MR4602735} in two ways: first, extending the statement from the triangle-free graph $\mathbb{H}_3$ to all $\mathbb{H}_{n+1}$, and second, carrying out the proof in $\mathsf{ACA}_0$ instead of $\mathsf{ACA}^+_0$. Liu and Patey \cite{LP} have also
recently shown the existence of $\mathbb{H}'$ in $\mathsf{ACA}_0$  for $n=2$ (without explicit bounds on the degree)  by a different route, showing that the Ordered Variable Word theorem is provable in $\mathsf{ACA}_0$ and applying a theorem in \cite{Hubicka_CS20}. By \cite{CDM},  for each $n\ge 2$ and  each $G$, where $|G|\geq 2$, there is a computable coloring $\chi$ so that any $\mathbb{H}'$ in which $\chi$ takes on at most  $k(G,n)$ colors computes $0^{(|G|-1)}$.  Thus,  for each $n\ge 2$,  the full statement ``$\mathbb{H}_{n+1}$ has finite big Ramsey degrees (for all $G$)'' is equivalent to $\mathsf{ACA}_0'$. 
We point out that  there is a gap between the lower bound $0^{(|G|-1)}$ found in \cite{CDM}
and our upper bound $0^{(2 \delta(G,n)-1)}$.
See Section 4 for further discussion.

Finite big Ramsey degree results for Henson graphs were first proved via coding tree arguments  developed  by the second author in \cite{MR4128725}  and further  extended in \cite{MR4568266}. 
In these arguments, one works with a representation of $\mathbb{H}_{n+1}$ as a certain subset of vertices in a binary tree, with the edge relation encoded in the branching of the tree. 
A key point is that  each finite subgraph  $G$ of 
$\mathbb{H}_{n+1}$
can be produced by only 
  finitely many configurations in this tree.
Set-theoretic forcing methods are applied to do unbounded searches for certain finite monochromatic sets, the proof taking place in the ground model, with the following consequence:
Given a finite coloring of the copies of $G$  in $\mathbb{H}_{n+1}$, we can pass to a subcopy of this tree in which the coloring is homogeneous on each of the finitely many configurations. 
By further passing to a copy of $\mathbb{H}_{n+1}$ which itself is a diary (as in \cite{Balko7}), 
there are at most $k(G,n)$ many colors:
 one for each diary representing $G$.

By now, several
 variants of  coding tree constructions and related forcing arguments have appeared in the literature, optimized for various applications (see \cite{CDP1, CDP2, CDW, Dobrinen_SDAP, DobrinenRado19, Dobrinen/Zucker23, MR4457343}).
In this paper, we use the version of coding tree from \cite{CDM}, which is a minor variation  of forms found in \cite{MR4128725, MR4568266}.
Section 2 presents  these
 coding trees  and diaries for Henson graphs.


In Section 3, we carry out the construction of monochromatic (on diaries) copies of $\mathbb{H}_{n+1}$. 
The main contribution here is 
Theorem \ref{thm.general}, 
which is a variant of the Halpern--\Lauchli\  Theorem suitable to coding trees for Henson graphs.
Our proof adapts  forcing arguments on coding trees (which use the Erd\H{o}s-Rado Theorem and hence  uncountable cardinals)
to the realm of computability theory.
This builds on the third author's arithmetical forcing proof of the Halpern--\Lauchli\ Theorem (unpublished).
Another novelty of our proof occurs in the induction steps of Lemma \ref{lem.3.12} and Theorem
\ref{thm.penultimatestep}.
A straightforward adaptation of the induction arguments in  \cite{MR4128725} and \cite{MR4568266}  would require  jumps of the form $0^{(n\cdot\omega)}$ for unbounded $n<\om$. 
Here, however, we dovetail the inductive applications of  Theorem \ref{thm.general}, yielding the following:

\begin{theorem}
[$\mathsf{ACA}_0$]\label{thm.upperbound}
Fix  $n\ge 2$.
Let $G$ be a finite subgraph of $\mathbb{H}_{n+1}$, and let $\delta(G,n)$ be the maximum height of a diary for $G$.  Then for any computable coloring of copies of $G$ in a computable copy of $\mathbb{H}_{n+1}$ into finitely many colors, there is a subcopy $\mathbb{H}'$ of $\mathbb{H}_{n+1}$ in which the color of $G$ is determined by its diary, and  $\mathbb{H}'$ is $0^{(2\delta(G,n)-1)}$ in the coloring.
\end{theorem}

It follows that 
 the statement,
``Henson graphs have finite big Ramsey degrees," is provable in ACA$_0'$.

\section{Coding Trees and Diaries for $\mathbb{H}_{n+1}$}\label{appendix}

This section provides background on coding trees and diaries.  
Original definitions can be found in 
\cite{MR4128725, MR4568266, Balko7}.
Here, we follow the  presentation in the Appendix of \cite{CDM}, reproducing 
 the minimum necessary for understanding our arguments in Section 3.

\begin{definition}[Standard Enumeration of $\mathbb{H}_{n+1}$]\label{defn.standardenum}
Fix $n\ge 2$. 
A {\em standard enumeration} of
$\mathbb{H}_{n+1}$  is a  computable enumeration of the vertices $\{\hat{v}_j:j<\om\}$ of 
$\mathbb{H}_{n+1}$ 
with the following additional property:
For each pair  $i<j<n$, $\hat{v}_j$ is connected to $\hat{v}_i$.
For each $j\ge n$, 
$\hat{v}_j$ forms an $n$-clique with its $n-1$ immediate predecessors; that is, 
all pairs of nodes among 
$\{\hat{v}_{j-n+1}, \dots,\hat{v}_j\}$ are connected. 
\end{definition}

The order on the nodes in  $\mathbb{H}_{n+1}$ induces a tree of binary sequences as follows:
For $j<\om$,  $\hat{v}_j$ is associated with   the
binary  sequence of length $j$, also denoted by $\hat{v}_j$, such that for $i<j$,  $\hat{v}_j(i)=1$ iff  $\hat{v}_j$ is connected to $\hat{v}_i$.
Thus, 
$\hat{v}_j(i)=0$ iff  
$\hat{v}_j$ is disconnected from $\hat{v}_i$.
Let $T\sse 2^{<\om}$ denote the collection of all initial segments of the binary sequences $\hat{v}_j$, $j<\om$.
Then  $T$ is a tree, ordered by initial segment  $\sse$.
The vertices of $\mathbb{H}_{n+1}$ appear  in $T$ as the nodes $\hat{v}_j$.
When referring to a  vertex  $\hat{v}_j$ as a  node in $T$, we will call it a {\em coding node} (as is common in the literature) or sometimes {\em vertex node} (as in \cite{CDM}).
We may also denote coding nodes by $c, c'$, etc.
Each subgraph $G=\{\hat{v}_j:j\in J\}\sse\mathbb{H}_{n+1}$ 
corresponds uniquely to the  subtree of $T$ consisting of all  initial segments of the vertex nodes $\{\hat{v}_j:j\in J\}$.
The tree $T$ 
has the following important 
property, which follows from the Extension Property of Henson graphs:
\begin{enumerate}
\item[$(*)$]
Given $k<\om$,
for each partition of 
$\{\hat{v}_j\mid j<k\}$ into two disjoint sets $V_0,V_1$  where $\mathbb{H}_{n+1}\upharpoonright V_1$ contains no $n$-clique,
there is a unique $s$ in $T$ of length $k$ so that for each $j\ge k$, 
$\hat{v}_j$ extends $s$ in $T$ iff 
$\hat{v}_j$ is disconnected with each vertex in $V_0$ and 
connected with every vertex in $V_1$.
\end{enumerate}


The tree structure $(T, \sse)$  with its lexicographic order and coding nodes is called a {\em coding tree} for $\mathbb{H}_{n+1}$ and forms the bedrock for the notion of diary.
The next notion involved in diaries are splitting nodes.
A node $s$ in a subset $S$ of $T$  is a {\em splitting node}  in $S$ iff both $s^{\frown}0$ and $s^{\frown}1$ are extended by some nodes in $S$. 
Diaries will correspond only to subgraphs of $\mathbb{H}_{n+1}$ in which each pair of its vertices $\hat{v}_j,\hat{v}_k$   splits, meaning that there is an $i<\min\{j,k\}$ so that exactly one of 
$\hat{v}_j(i),\hat{v}_k(j)$   equals one.

The third notion involved in diaries is that of  consecutive age-changes.
A set $X\sse T$ is called a {\em level set} iff $X\sse 2^\ell$ for some $\ell<\om$.
Given a level set $X$, let  $\ell(X)$ denote the length of its nodes.
When we write a level set as $X=\{x_i\mid i<p\}$, it is assumed that the $x_i$ are all distinct, so $X$ has size $p$.
For $\ell'<\ell(X)$, $X\upharpoonright\ell'$ denotes $\{x\upharpoonright\ell'\mid x\in X\}$, the collection of nodes in $X$ truncated to  length $\ell'$.
This of course may have smaller cardinality than $X$.

\begin{definition}\label{defn.Kmp}
For $1\le m< n$,
we define $K(m)$ to consist of those level sets 
 $X=\{x_i\mid i<p\}\sse T$ for which there is 
 an  $m$-clique  $\{\hat{v}_{k_q}\mid q<m\}$
where 
$k_0<\dots<k_{m-1}$, 
$\ell(X)=k_{m-1}+1$,  and 
for each $i<p$ and $q<m$, $x_i(k_q)=1$.
We say that such an  $m$-clique  {\em witnesses} that $X$ is in $K(m)$.

A level set $X=\{x_i\mid i<p\}$ is in $\AC(m,p)$ iff $X$ is in $K(m)$ and
for all $\ell'<\ell(X)$, $X\upharpoonright\ell'\not\in K(m)$.
We say that $X$ is an {\em age-change} iff $X\in \AC(m,p)$ for some $1\le m< n$  and  $1\le p  \le n+1-m$.
\end{definition}

There are no age-changes for $p>n+1-m$ (see \cite{Balko7}): such $p$-tuples are not a part of diaries.
Notice that if an $m$-clique  $\{\hat{v}_{k_q}\mid q<m\}$  witnesses that $X$ is in $K(m)$, then 
 any vertex  $\hat{v}_j$ 
 extending any node  in $X$ is connected to every member of the $m$-clique $\{\hat{v}_{k_q}\mid q<m\}$.
As $m$ increases, the possible connections between vertices in $\mathbb{H}_{n+1}$ extending the members of $X$ decreases. 
In particular, if $X$ is in $K(n-1)$, then 
any  collection of vertices in $\mathbb{H}_{n+1}$  extending members of $X$ must form an anticlique, since each such vertex forms an $n$-clique with $\{\hat{v}_{k_q}\mid q<n-1\}$.

\begin{definition}[Consecutive age-change]\label{defn.conagechange}
For the following, 
$1\le m<n$ and 
 $1\le p  \le n+1-m$.
For $p=1$, for each $1\le m< n$ 
we define $\Con(m,1)=\AC(m,1)$;  
every  age-change  in $\AC(m,1)$ is {\em consecutive}.
For  $p\ge 2$,
we define $X\in \AC(m,p)$ to be a 
{\em consecutive age-change}, and write $X\in\Con(m,p)$,
 iff   
the following hold:
\begin{enumerate}
\item
For  each proper subset $Y\subset X$,
there is some (unique) $\ell<\ell(X)$ so that $Y\upharpoonright \ell$ is  
in $\Con(m,p')$, where $p'$ is the size  of $Y\upharpoonright \ell$.
\item
If $m\ge 2$, 
then for each $1\le m'<m$ and
 each $Y\sse X$ 
there is an $\ell<\ell(X)$ such that
 $Y\upharpoonright \ell$ is in $\Con(m',p')$, where $p'$ is the size of  $Y\upharpoonright \ell$.
\end{enumerate}
\end{definition}

The last ingredient of diaries for Henson graphs is called ``controlled coding triples"  \cite{Balko7}.
 These add as many age-changes as possible, subject to the constraint of 
 not affecting the graph structure (of the  subgraph of $\mathbb{H}_{n+1}$  encoded by the diary). 
As the precise definition  of controlled coding triple is somewhat involved for $n\ge 3$ and not used in our proofs, 
 the interested reader is referred to \cite{Balko7} or \cite{CDM}.
For our present work, it suffices to know that that for each finite subgraph $G$ of $\mathbb{H}_{n+1}$, there are finitely many diaries associated with copies of $G$ in $\mathbb{H}_{n+1}$, and  each diary is characterized by the tree structure along with the placement of splitting nodes, age-changes, and  vertices.

We say that a node in $T$ is in the {\em spine} of $T$ if it is a finite sequence of $0$'s.
The subtree of $T$  induced by any subcopy of 
$\mathbb{H}_{n+1}$  retains the spine.
It follows from our standard enumeration of $\mathbb{H}_{n+1}$ that  $\hat{v}_j$ is in $\Con(n-1,1)$, for each $j\ge n-1$.
The effects of this are that (for all $j\ge n-1$) 
the  coding node $\hat{v}_j$ in $T$ is off the spine and   has exactly one immediate successor, namely 
${\hat{v}_j}^{\frown}0$. 
Without loss of generality, we may work in the copy $\mathbb{H}_{n+1}\re\{\hat{v}_j:j\ge n-1\}$ and assume these properties for all  coding nodes. 
We now discuss diaries for Henson graphs.

\begin{definition}\label{def.diary}
Let $G=\{\hat{v}_{j_i}\mid i<g\}$, ($1\le g\le \om$),  be  a subgraph of $\mathbb{H}_{n+1}$.  
We say that  $G$  {\em induces  a   diary}  iff 
letting $A=A(G)$ denote the 
meet-closure  in $T$ of the vertices $\{\hat{v}_{j_i}\mid i<g\}$,
the following hold:
\begin{enumerate}
\item
Each pair of  vertices  in $G$ splits.
Hence,  the vertices in $G$  are exactly the terminal nodes in $A$.
\item 
Each age-change in $A$ is consecutive.
\item
Each level of $A$ has at most one of three types of 
events:
\begin{enumerate}
\item[(a)]
Exactly one splitting node.
This can either be the meet of a vertex node with the spine, or a splitting node that is the meet of two vertex nodes in $A$.
Moreover, except for the one node that extends the splitting node by $1$, all other nodes at this level extend by $0$ (ensuring there are no age-changes, except  when the splitting node  is splitting off from the spine). 
\item[(b)]
Exactly one  consecutive age-change. In this case, the set of nodes involved in the consecutive age-change and the  type of age-change, $\Con(m,p)$,   are recorded.
\item[(c)]
A vertex, say $\hat{v}_{j_i}$. In this case, for each $i<k<g$, the value of $\hat{v}_{j_k}(|\hat{v}_{j_i}|)$ is recorded.
\end{enumerate}
The {\em interesting levels} of $A$
are those levels 
in which an event of {\em type} (a), (b), or (c) takes place.
\item
In (3c),  for $i<g$, 
the collection of nodes  $\{\hat{v}_{j_k}:i<k<g\}$ 
forms a controlled coding triple.
This shows up as consecutive age-changes as part of (3b).
\end{enumerate}

 Let  $L(A)$ denote the set of those $\ell$ for which $A\upharpoonright\ell$ is an  interesting level of $A$, and let $L=|L(A)|$.
If  $G$  induces a diary, then  the {\em diary}
 of $G$,  $\Delta=\Delta(G)$,  is defined to be 
the subtree of $2^{< L}$ 
that is isomorphic (in both  the tree and lexicographic orders) to
$\bigcup_{\ell\in L(A)} A\upharpoonright\ell$,
 along with  the following information:
 the type of each level of interest,
 the node(s) in the level  involved in the type,
 and in the case of a vertex node,
 the  value ($0$ or $1$) of the immediate successors of all non-vertex nodes in that level
 (equivalently, the edge/non-edge relations between the vertices in $G$).
Two graphs $F$ and $G$ are said to {\em have the same diary} 
iff $\Delta(F)$ and $\Delta(G)$  are equal.
\end{definition}

The number of diaries for a given finite graph $G$ is finite.
Further,  if $G\sse T$ induces a diary, then every subgraph of $G$ also induces a diary. (See \cite{Balko7} or \cite{CDM}.)
\begin{definition}\label{defn.deltaG}
   Given $n\ge 2$ and  a finite subgraph $G$ of $\mathbb{H}_{n+1}$,
we define $\delta(G,n)$ to be the maximum number of levels in the diaries for $G$ in $\mathbb{H}_{n+1}$.
\end{definition}
A very rough upper bound for 
 $\delta(G,n)$ is 
\begin{equation}
2|G|(n-1)\sum_{1\le k\le n}{|G|\choose k}
\end{equation}
Here, $2|G|$ is the number of levels in a Joyce tree with $|G|$ many terminal nodes.
Between each two levels of such a tree, there can be at most 
$\sum_{1\le k\le n}{|G|\choose k}$ many ways to choose sets of nodes for an age-change, and at most $n-1$ many different age changes for each of these sets of nodes.

An {\em embedding} of  a diary $\Delta$ into $T$ is a tree, lexicographic order, coding node, and relative-length preserving injection $e:\Delta\ra T$ such that the 
$e$-image of the terminal nodes  in $\Delta$ 
is a set of vertex nodes for a graph 
$G$ that induces $\Delta$.
(It follows that $e$ also preserves age-changes.)
A  {\em copy} of $\Delta$ in $T$  is the image of some embedding of $\Delta$ into $T$.

\section{Big Ramsey degrees in Henson graphs in ACA$_0$}

In this section, 
we find bounds on the number of jumps needed to prove that a given finite diary has the Ramsey property. 
To do so, we recast
arguments in \cite{MR4128725} and \cite{MR4568266}  that used set-theoretic forcing on coding trees into the arithmetic setting.

The tree $T$ in the previous section will be our template tree.
We say that a subtree $S\sse T$ which encodes a copy of $\mathbb{H}_{n+1}$ is {\em aged-isomorphic} 
to $T$ iff 
there is a  tree-order and lexicographic-order preserving  bijection $f:T\ra S$ satisfying the following:
\begin{enumerate}
\item
$f$ takes the $i$-th coding node  $\hat{v}_i$ in $T$ to the $i$-th coding node (in order of length)  in $S$.
The $i$-th level of $S$, denoted $\Lev_S(i)$, is defined to be the set of those $s\in S$ with $|s|=|f(\hat{v}_i)|$.
\item
$f$ takes the $i$-th level of $T$ to the $i$-th level of $S$, preserving the lexicographic order.
\item
For each node $t\in T$, $t$ splits in $T$ iff $f(t)$ splits in $S$.
\item
Between levels of $S$  there are no splitting nodes, no coding nodes,  and no age-changes (except for the age-change at the immediate successors (in $T$) of the levels in $S$). 
\end{enumerate}

By an {\em initial segment} of a diary $\Delta$ we mean $\Delta\re\ell:=\{s\in\Delta:|s|<\ell\}$ for some   $\ell\le |\Delta|$, where $|\Delta|$ denotes the  number of levels in $\Delta$.
In particular, if $\ell= |\Delta|$ then $\Delta\re\ell=\Delta$, and if 
$\ell=0$ then 
$\Delta\re\ell$ is the empty set.
The {\em immediate predecessor} of $\Delta\re\ell$ is $\Delta\re(\ell-1)$.
A {\em copy} of $\Delta\re\ell$ is the $e$-image of $\Delta\re \ell$ for some embedding $e$ of $\Delta$ into $T$.

Suppose we have some  non-empty initial segment of a finite diary, which (abusing notation) we shall also denote by $\Delta$,
and a coloring of all copies of $\Delta$ in $T$. Let $\Delta^-$ denote the immediate predecessor of $\Delta$. 
We  include the possibility that 
 $\Delta$ is  the minimal node of a diary. 
In that case,
 $\Delta^-$ is the empty set and 
$\Delta$ is a splitting node on the spine.
An intermediate step  towards finding a subtree in which all copies of a finite diary have the same color is to find an aged-isomorphic subtree $S$ of $T$ so that  the color of the copies of $\Delta$ in $S$ is entirely determined by which copy of $\Delta^-$ they extend.
A predecessor step to that is 
the content of the next theorem, whose proof is the main technical step.

\begin{theorem}[$\mathsf{ACA}_0$]\label{thm.general}
Let $\tilde{\Delta}$ be a finite diary, $\Delta$ an initial segment of $\tilde{\Delta}$, and $\Delta^-$ the immediate  predecessor of $\Delta^-$ in $\tilde{\Delta}$.
Let  $A$ be a proper initial segment of $T$, 
and suppose $D$ is a copy of $\Delta^-$ in $A$ whose maximum level is contained in the maximum level of $A$. 
Fix  a coloring $\chi$ of all copies of $\Delta$ in $T$ extending $D$.
Then there is a computable (in $T$) aged-isomorphic subtree   $S$ of $T$
extending $A$ such that 
every extension of $D$ to 
a copy of $\Delta$ in $S$ has the same $\chi$-color.
Further, an index of this subcopy can be computed by $0''$  from $\Delta, A, D$ and a code for $T$.
In the case that
$\Delta$ is a single splitting node on the spine, an index can be computed by $0'$ from $\Delta, A, D$ and a code for $T$.
\end{theorem}

We first prove 
 Theorem \ref{thm.general} for the special case when  the maximum level of $\Delta$ has only one coding node, as its
 proof is much simpler.
After that,  we will give the proof of  Theorem \ref{thm.general} for all other cases.

\begin{lemma}[$\mathsf{ACA}_0$]\label{lem.endhomonecn}
Theorem \ref{thm.general} holds when the maximum level of $\Delta$ has only one coding node (and therefore $\Delta$ is the diary $\tilde{\Delta}$).
\end{lemma}

\begin{proof}
In this case, 
the maximum level of $D$   has at most two nodes, only one of which extends to a coding node to make a copy of $\Delta$, denoted by $x$.
Let $x^+$ denote the immediate successor of $x$ such that the union of $D$ with  any coding node extending $x^+$ 
makes a copy of $\Delta$.
Let  $\Phi(x^+,D)$ be the sentence:
$$
\forall s\contains x^+\ \exists c\contains s \ \chi(D\cup\{c\})= 0.
$$
Thus $\Phi(x^+,D)$ says that densely above $x^+$, there are coding nodes which, when unioned with $D$, all have $\chi$-color $0$.
The negation $\neg\Phi(x^+,D)$ says
$$
\exists s\contains x^+\ \forall c\contains s \ \chi(D\cup\{c\})= 1.
$$
That is, there is an $s$ extending $x^+$ so that every coding node in the cone above $s$, when unioned with $D$, has $\chi$-color $1$.
We build an aged-isomorphic subtree $S$ of $T$ (end-extending $A$) where all copies of $\Delta$ extending $D$ have the same color.

Suppose  $\Phi(x^+,D)$ holds and we have built $S$ extending $A$ up to level $m-1$, call this $A_{m-1}$, where $m$ is such that 
$\Lev_T(m)$ has a coding node extending $x^+$.
(We assume that $A_{m-1}$ is aged-isomorphic to $T$ truncated to its level $m-1$.)
Let $s$ denote the 
node in 
$\Lev_S^+(m-1)$ that must be extended to a coding node in any aged-isomorphic copy of $T$ up to $\Lev_T(m)$.
Note that $s$ extends $x^+$, so 
by $\Phi(x^+,D)$, there is some coding node $c$ extending $s$ so that $\chi(D\cup\{c\})=0$.
Extend all other nodes in  $\Lev_S^+(m-1)$ leftmost to the length of $c$.  The union of these nodes with $c$ is $\Lev_S(m)$.

Suppose  $\neg\Phi(x^+,D)$ holds. 
First extend $x^+$ to a node $s$ so that all coding nodes in the cone above $s$ give copies of $\Delta$ with $\chi$-color $1$.
Extend all other nodes in $\Lev_S^+(m-1)$ leftmost to the same length as $s$.
Call this level set $Z$; it will not be a part of $S$ but every node in $S$ above $A$ will extend some node in $Z$.
From here on, build $S$ level by level so that it is an aged-isomorphic copy of $T$.  Every coding node in $S$ that is part of a copy of $\Delta$ extending $D$ will automatically extend $s$ and hence give $\chi$-color $1$ when unioned with $D$.
\end{proof}


\subsection{Proof of Theorem \ref{thm.general}}

This subsection is devoted to proving the remaining cases of Theorem \ref{thm.general}:
 either the maximum level of $\Delta$ has at least two nodes or it has one non-coding node which has a consecutive age-change.

Let $X$ denote the nodes of maximum length in $D$.
There are three possibilities:
$X$ could have a splitting node, a coding node, or a consecutive age-change. 
In each of these cases we 
consider  as well   the immediate successors in $T$ of any (equivalently, every) extension of $X$ to a copy of $\Delta$. 
If $X$ has a splitting node, then  let $X^+$ consist of  both  immediate successors  of the splitting node  in $X$ and 
extend all other nodes in  $X$ with exactly one $0$.
If $X$ has a coding node, 
then  $X^+$ consists of  $x^+$ for each non-coding node $x\in X$, where 
$x^+=x^{\frown}0$
  iff every extension of $D$ to a copy of $\Delta$ immediately extends $x$ by $0$; 
otherwise  $x^+=x^{\frown}1$.
If $X$ has a consecutive age-change,  then let $X^+$ be $\{x^+\mid x\in X\}$, where 
for each $x\in X$, $x^+=x^{\frown}0$ iff every extension of $D$ to a copy of $\Delta$ immediately extends $x$ by $0$;
otherwise $x^+=x^{\frown}1$.
In particular, the nodes in $X^+$ all have length $\ell(X)+1$.
Note that  
every extension of $D$ to a copy of $\Delta$ in $T$ has exactly
$X^+$ as its truncation to length $\ell(X)+1$.

Enumerate $X^+$ as $\{x_0,\dots,x_d\}$, where the order will be determined  in the following paragraphs.
Let $\mathbf{Y}$ be any level set in $T$ end-extending $X^+$  to a 
copy of $\Delta$ in $T$.
There are no age-changes between $X^+$ and  $\mathbf{Y}$.
Recall that in this lemma, we are assuming that the maximum level of $\Delta$ has at least two nodes,  so $\Delta\ne \tilde{\Delta}$.
Let 
 $\mathbf{Y}^+$ denote the immediate successors of $\mathbf{Y}$ in some/any extension of $D\cup \mathbf{Y}$ to a copy of $\tilde{\Delta}$.  We now have three cases:

\underline{Cases  1 and 2}.
If the maximum level of $\Delta$ has either (1) a splitting node or  (2) a coding node, let $x_d$ denote  the unique node in $X^+$ that must extend to a splitting or coding node in order to make a copy of $\Delta$.
Let $\mathbf{y}_d$ denote the splitting or coding node in $\mathbf{Y}$ (so $\mathbf{y}_d$ extends $x_d$).
In Case 1, let $\mathbf{Y}^+$
consist of the immediate successors of $\mathbf{y}_d$,
 $\mathbf{y}_d^+={\mathbf{y}_d}^{\frown}0$ and  $\mathbf{y}_{d+1}^+={\mathbf{y}_d}^{\frown}1$, as well as  $\mathbf{y}_i^+={\mathbf{y}_i}^{\frown}0$ for all $i<d$.
In Case 2, let $e\le d$ be the number of nodes in $\mathbf{Y}^+$ that have last entry $1$; let $\mathbf{y}_i$, $i<e$, enumerate the nodes in $\mathbf{Y}^+$ with last entry $1$,
and let $\mathbf{y}_e,\dots, \mathbf{y}_{d-1}$ enumerate the nodes in $\mathbf{Y}^+$ with last entry $0$.
In both cases, for each $i<d$, let $\mathbf{y}_i$ denote the immediate predecessor of $\mathbf{y}_i^+$ so that $\mathbf{Y}=\{\mathbf{y}_i\mid i\le d\}$.
For all $i\le d$, let $x_i$ denote the predecessor of $\mathbf{y}_i$ in $X^+$.

\underline{Case 3}.  The maximum level of $\Delta$ has a consecutive age-change over $\Delta^-$. 
Then 
let $e\le d$ and the indexing of $X^+$ be such that 
any copy of $\Delta$ extending $X^+$
must have
the nodes $x_0,\dots,x_{e-1}$
 extend by $1$
and the nodes $x_e, \dots, x_d$  extend by $0$.
Then for each $i\le d$, let $\mathbf{y}_i$ denote the node in $\mathbf{Y}$ extending $x_i$, and let $\mathbf{y}_i^+$ denote the node in $Y^+$ extending $\mathbf{y}_i$ (and hence, also extending $x_i$).
Note that  the last entry of $\mathbf{y}_i^+$ is a $1$ iff $i\le e$.

This completes the set-up enabling us to  define the forcings.
Let $A$ denote the tree $T$ truncated to the level containing $X^+$.
Let $\mathcal{Y}$ denote the collection of all level sets $Y$ end-extending $X^+$ that 
extend $D$ to a 
 copy of $\Delta$.
In each of the three cases, if $Y$ is in $\mathcal{Y}$, then the collection of  the immediate successors of $Y$ in $T$ that extends to a copy of $\tilde{\Delta}$ is uniquely determined by $Y$; call this set   $Y^+$.
 Note that for each $Y\in\mathcal{Y}$,
  $X\cup Y^+$  aged-isomorphic to
$X\cup \mathbf{Y}^+$.

Let $\chi:\mathcal{Y}\ra \{0,1\}$ be a coloring.
Note that  $\chi$ uniquely determines a coloring on the set  $\{Y^+\mid Y\in\mathcal{Y}\}$.  
For each $j\le d$, let $T_j$ denote the cone in $T$ above $x_j$.
Each of the three cases will involve a different  forcing, whose definitions start the same.

An {\em ordinary condition}  is a function  $p$ with the following properties.
The domain of $p$ is a finite subset of $(d+1)\times\om$
and range some level set in $T$, where for each $(i,n)$ in the domain of $p$, $p(i,n)\in T_i$.
In particular, this implies that $p(i,n)$
extends $x_i$.
In Case 1,
we require that 
$p(d,n)$ is a splitting node in $T_d$, for each $n$ for which $(d,n)$ is in the domain of $p$.
In Case 2, we require that $p(d,n)$ is a coding node  in $T_d$, for each $n$ for which $(d,n)$ is in the domain of $p$.
There is only one coding node in each level of $T$, so all  $p(d,n)$ will be the same coding node
in this case. For each $i<e$, we further require that $p(i,n)$ is a splitting node in $T_i$ (so that it will have an immediate extension by $1$).
For $e\le i<d$, $p(i,n)$ may be either splitting or non-splitting--it will not matter.
In Case 3,  for each $i<e$, we require that $p(i,n)$ is a splitting node in $T_i$.
For $e\le i\le d$, $p(i,n)$ may be either splitting or non-splitting--it will not matter.

This completes the definition of the conditions for the three cases. 
For $k\in \{1,2,3\}$, let $P_k$ denote the set of all conditions for Case $k$.
The partial ordering  $\le$ is defined in the same way on each $P_k$:
For $p,q\in P_k$,  $q\le p$ holds iff 
\begin{enumerate}
\item $\dom(p)\subseteq\dom(q)$,
    \item
for each $(i,n)\in \dom(p)$, $q(i,n)\contains p(i,n)$, and 
\item
$\ran(q)\re \dom(p)$ has no age-changes over $\ran(p)$.
\end{enumerate}
The important point of the definition of 
$(P_k,\le)$  is that, for  $k\in\{1,2,3\}$
and  $p,q\in P_k$, 
 if $\{p(i,n_i)\mid i\le d\}$ is a member of $\mathcal{Y}$  and $q\le p$
 then $\{q(i,n_i)\mid i\le d\}$ is also a member of $\mathcal{Y}$.

 From now on, fix $k\in\{1,2,3\}$ and let $P$ denote $P_k$.
The goal now is to use $(P,\le)$ to construct an aged-isomorphic subtree $S$ of $T$, extending $A$, so that each member of $\mathcal{Y}$ in $S$ has the same color.
That is, to construct an $S$ that is homogeneous for all members of $\mathcal{Y}$ in $S$.
The plan of construction is as follows:
First, we use the forcing partial order to end-extend the nodes in $X^+$ to some set of nodes for which it is always possible to extend that set to members of $\mathcal{Y}$ with the same color. 
After that,
we alternate between utilizing the forcing  partial order at  levels  that are immediate predecessors of levels in $S$ which must
contain members of $\mathcal{Y}$, and building  all other levels of the subtree by hand.


\medskip

Given a condition $p\in P$, 
let $\ell_p$ denote the length of the nodes in the range of $p$.
The following convention will be handy.
When we have a sequence $\vec{i}=(i_0,\dots,i_d)\in\om^{d+1}$ and a condition $p$ so that for each $j\le d$, $(j,i_j)\in\dom(p)$, 
we will only write
$p(\vec{i})$ 
to denote the tuple $(p(0,i_0),\dots,
p(d,i_{d}))$ when that tuple is an element of $\mathcal{Y}$.
In particular, we write $p(\vec{i})$ only when  $\chi(p(\vec{i}))$ is defined.

\begin{definition}
 A {\em simple configuration} $\cC=(I^{\cC}, \{p^{\cC}_{\vec{i}},\epsilon^{\cC}_{\vec{i}}\}_{\vec{i}\in I^{\cC}})$ consists of $I^\cC\sse \omega^{d+1}$, and for each $\vec{i}\in I^\cC$:
 \begin{enumerate}
\item 
$p_{\vec i}$  is a condition such that for each $j\le d$, $(j,i_j)\in\dom(p_{\vec i})$, and 
the tuple $p_{\vec i}(\vec{i})$ is an element of $\mathcal{Y}$;
\item 
a color $\epsilon_{\vec i}\in 2$. 
 \end{enumerate}
When $\pi\in(\omega\rightarrow\omega)^{d+1}$ is a tuple of injective functions, we write $\cD\le_\pi \cC$ if
\begin{enumerate}
    \item 
$\pi[I^{\cC}]\sse I^{\cD}$;
    \item
    for each $\vec i\in I^{\cC}$, $p^{\cD}_{\pi(\vec i)} = p^{\cC}_{\vec i}$ and  $\epsilon^{\cD}_{\pi(\vec i)} = \epsilon^{\cC}_{\vec i}$, where $\pi(\vec i)=(\pi_0(i_0),\ldots,\pi_d(i_d))$.
\end{enumerate}
We write $\cD\le\cC$ if there is some $\pi$ so that $\cD\le_\pi\cC$.

A simple configuration $\cC$ is {\em consistent} if whenever $J\sse I^\cC$ is finite and $q$ is a condition such that  $q\le p_{\vec i}$ for each $\vec i\in J$,
then 
there are infinitely many $\ell\in\om$ for which there is an $r\le q$ with $\ell_r=\ell$ and for each $\vec i\in J$, $\chi(r(\vec i))=\epsilon^{\cC}_{\vec i}$.
\end{definition}

The empty configuration is trivially consistent.  We need to ensure that we can construct consistent configurations with arbitrary domains.

\begin{lemma}[$\mathsf{ACA}_0$]
    If $\cC$ is a consistent simple configuration, $\pi\in(\omega\rightarrow\omega)^{d+1}$ is a tuple of injective functions, and $\vec{i}^*\in\omega^{d+1}$ then there is a consistent simple configuration $\cD\le_\pi \cC$ with $\vec{i}^*\in I^{\cD}$.
\end{lemma}
    
\begin{proof}
 Let $\cC_0$ be the largest condition $\le_\pi\cC$---that is, the condition determined by $I^{\cC_0}=\pi[I^{\cC}]$. If $\vec{i}^*\in I^{\cC_0}$ then we are done, so we may assume $\vec {i}^*\not\in I^{\cC_0}$. Let $\cC'$ be the simple configuration extending $\cC$ by setting $I^{\cC'}=I^{\cC_0}\cup\{\vec{i}^*\}$ and taking $p^{\cC'}_{\vec{i}^*}$ to be the condition with $\dom(p^{\cC'}_{\vec{i}^*})=\{(j,i^*_j)\mid j\le d\}$, and $p^{\cC'}_{\vec{i}^*}(j,i^*_j)=x_j$  for all $j\le d$, and $\epsilon^{\cC'}_{\vec{i}^*}=0$.

If $\cC'$ is consistent, we are done, so suppose otherwise. 
Then there is a finite set $J\sse I^{\cC'}$ and a condition $q^*$ witnessing that $\cC'$ is inconsistent. 
This literally means that 
 $q^*\le p_{\vec i}$ for each $\vec i\in J$ and
 for all but finitely many $\ell$, for all $r\le q^*$ with $\ell_r=\ell$, there is an $\vec{i}\in J$ such that $\chi(r(\vec{i}))\ne\epsilon^{\cC}_{\vec{i}}$.
Hence, 
there are only finitely many 
$\ell\in\om$ for which there is an 
$r\le q^*$ with $\ell_r=\ell$  such that 
 for each  $\vec i\in J$,
$\chi(r(\vec i))=\epsilon^{\cC'}_{\vec i}$.
Since $\cC$, and so also $\cC_0$, is consistent, we must have that $\vec i^*\in J$.

Let $\cD$ be the simple configuration extending $\cC_0$ with $\vec i^*\in I^\cD$,  $p^{\cD}_{\vec i^*}=q^*$, and $\epsilon^{\cD}_{\vec i^*}=1$.  We claim that $\cD$ is consistent. 
Consider any finite $K\sse I^\cD$
and any $q\le p_{\vec i}^{\mathcal{D}}$ for all $\vec i\in K$.
If $\vec i^*\not\in K$, then $K\sse I^{\cC_0}$, so the existence of an $r$ witnessing its consistency follows from the consistency of $\cC_0$.
Now suppose that  $\vec i^*\in K$.  Then $q\le q^*$, and 
it follows that $q\le p_{\vec i}$ for all $\vec i\in J$ as well.
By the consistency of $\cC$, there are infinitely many levels $\ell$ for which there is an $r\le q$ so that $\ell_r=\ell$ and for each $\vec i\in (J\cup K)\setminus \{\vec i^*\}$, $\chi(r(\vec i))= \epsilon^{\cC}_{\vec i}$.
Therefore, except for finitely many of these levels, we have $\chi(r(\vec i^*))=1=\epsilon^{\cD}_{\vec i}$, as needed.
\end{proof}

If a condition $p$ is defined on a large domain (relative to $\ell_p$) it could be that there are many different $n\in\om$ with  $p(j,n)$ being the same node in the tree $T_j$.
We need other conditions which allow us to describe conditions which have `unboundedly many' copies of the same value.

\begin{definition}\label{defn.7.5}
A {\em condition} in $P^{\infty}$ is a function $p$ such that 
\begin{enumerate}
\item 
$\dom(p)$ is a finite subset of $(d+1)\times (\om\cup\{\infty\})$;
\item 
Replacing each occurrence of the form $(j,\infty)$ in $\dom(p)$ by  any $(j,n)$, where  $n\in\om$ and $(j,n)\not\in\dom(p)$,
produces an ordinary condition in $P$.
\end{enumerate}
We partially order $P^{\infty}$ as follows:
For $p,q\in P^{\infty}$, define $q\le p$ iff 
\begin{enumerate}
\item $\dom(p)\subseteq\dom(q)$;
    \item
for each $s\in \dom(p)$, $q(s)\contains p(s)$;
\item
$\ran(q)\re \dom(p)$ has no age-changes over $\ran(p)$.
\end{enumerate}
\end{definition}


Roughly speaking, $p(j,\infty)=t$ in a configuration will mean that we can have arbitrarily large (finite) sets $N\sse \om$ so that $p(j,n)=t$ for all $n\in N$. To make this precise, we need the idea of an \emph{$\vec i$-instantiation} of a condition $p$, which is an ordinary condition  $p^{\infty\mapsto \vec i}$ 
obtained by replacing $(j,\infty)$ with $(j,i_j)$ in the domain and 
defining 
 $p^{\infty\mapsto \vec i}(j, i_j)=p(j,\infty)$.

\begin{definition}
Let  $p\in P^{\infty}$ be a  condition.
Suppose
$\vec i=(i_0,\dots,i_d)\in \om^{d+1}$.
The {\em $\vec i$-instantiation of $p$} is an ordinary condition given by:
\begin{enumerate}
    \item $\dom(p^{\infty\mapsto\vec i})=(\{(j,n)\in\dom(p)\mid n\in\omega\})\cup (\{(j,i_j)\mid (j,\infty)\in\dom(p)\})$;
\item for $(j,n)\in\dom(p)$ with $n\in\omega$, $p^{\infty\mapsto\vec i}(j,n)=p(j,n)$;
\item for $(j,\infty)\in\dom(p)$, $p^{\infty\mapsto\vec i}(j,i_j)=p(j,\infty)$.
\end{enumerate}
\end{definition}

The  next notion of  `configuration' extends the notion of `simple configuration' above.

\begin{definition}\label{def.configuration}
 A {\em  configuration} $\cC=(I^{\cC}, \{p^{\cC}_{\vec{i}},\epsilon^{\cC}_{\vec{i}}\}_{\vec{i}\in I^{\cC}})$ consists of an index set $I^\cC\sse (d+1)\times(\om\cup\{\infty\})$, and for each $\vec{i}\in I^\cC$:
 \begin{enumerate}
\item 
a condition $p_{\vec i}$ such that for each $j\le d$, $(j,i_j)\in\dom(p_{\vec i})$ and $(j,\infty)\in \dom(p_{\vec i})$ only when $i_j=\infty$;
\item $\{p_{\vec i}(j,i_j)\mid j\le d\}$ is in $\mathcal{Y}$;
\item 
a color $\epsilon_{\vec i}\in 2$. 
 \end{enumerate}
When $\pi\in(\omega\rightarrow\omega)^{d+1}$ is a tuple of injective functions, by abuse of notation we also write $\pi$ for the tuple of functions from $\omega\cup\{\infty\}$ to $\omega\cup\{\infty\}$ which extends $\pi$ by setting $\pi_i(\infty)=\infty$ for all $i$. We write $\cD\le_\pi \cC$ if 
\begin{enumerate}
    \item 
$\pi[I^{\cC}]\sse I^{\cD}$;
    \item
    for each $\vec i\in I^{\cC}$, $p^{\cD}_{\pi(\vec i)} = p^{\cC}_{\vec i}$, and  $\epsilon^{\cD}_{\pi(\vec i)} = \epsilon^{\cC}_{\vec i}$.
\end{enumerate}
We write $\cD\le\cC$ if there is some injective $\pi:\omega\rightarrow\omega$ so that $\cD\le_\pi\cC$.
\end{definition}

Next we need to define when a configuration is consistent.  The crucial complication is that it is not enough to just consider common extensions of finitely many conditions: when we have $\infty$ in coordinates, we need to be able to make many copies of them.

\begin{definition}
When $\cC$ is a configuration, an {\em instantiation} is a
triple $(K,q,\rho)$ where 
 $K\sse \om^{d+1}$ is finite,  $q$ is an ordinary condition, and $\rho:\dom(q)\ra (d+1)\times (\om\cup\{\infty\})$ is  a function  satisfying the following:
\begin{enumerate}
    \item 
For each $\vec i\in K$ and each $j\le d$, $(j,i_j)\in\dom(q)$;
\item 
For each $(j,n)\in\dom(\rho)$,
$\rho(j,n)\in\{(j,n),(j,\infty)\}$;
\item 
For each $\vec i\in K$, $\rho(\vec i):=(\rho(j,i_j):{j\le d}) \in I^\cC$;
\item
For each $\vec i\in K$, $q\le (p^\cC_{\rho(\vec i)})^{\infty\mapsto \vec i}$.
\end{enumerate}

A configuration is {\em consistent} if whenever $(K,q,\rho)$ is an instantiation of $\cC$, there are infinitely many $\ell<\om$ for which there is an $r\le q$ with $\ell_r=\ell$ and $f(r(\vec i))=\epsilon^\cC_{\rho(\vec i)}$ for each $\vec i\in K$.
\end{definition}
When $\cC$ is a simple configuration, this is the same definition as before: the set $K$ picks out some subset of $I^{\cC}$, and $q$ must extend $p_{\vec{i}}^{\cC}$ for each $\vec i\in K$; in this case $\rho$ is simply the identity. In all cases, $K$ and $q$ contain only coordinates in $\omega$.

A simple motivating example of the generalization is when $d=1$ and $I^{\cC}=\{(\infty,\infty)\}$. Then we may take $K$ to be any finite set of pairs---say, $\{0,1\}\times[0,7]$---and $\rho(0,n)=(0,\infty)$ and $\rho(1,n)=(1,\infty)$ for all $n\in[0,7]$.
Then an instantiation is an ordinary condition $q$ such that $\{0,1\}\times[0,7]\subseteq\dom(q)$ and for all $(i_0,i_1)\in \{0,1\}\times[0,7]$, we have $q_{(i_0,i_1)}\leq p^{\cC}_{(\infty,\infty)}$.  Note that the $q_{(i_0,i_1)}$ need not be the same: saying $\cC$ is consistent is saying that for \emph{any} $q$, which is a finite grid of extensions of $p^{\cC}_{(\infty,\infty)}$, we may find infinitely many levels with extensions in the grid in the correct color.
That is, the $\infty$ coordinate in $I^{\cC}$ can become a large finite set $\rho^{-1}(j,\infty)$, which corresponds to making many copies of the condition $p^{\cC}_{(j,\infty)}$.

In the following, let 
$\vec{\infty}$ denote $(\infty_0,\dots,\infty_d)$.
Recall that part  (2) in Definition \ref{defn.7.5} implies that  if $p\in P^{\infty}$ and $(j,\infty)\in\dom(p)$, then $p(j,\infty)\in T_j$.
We now show that consistent configurations  with $\vec{\infty}\in I^\cC$ suffice to prove Theorem \ref{thm.general}.

\begin{lemma}[$\mathsf{ACA}_0$]\label{lem.Dendhomog}
Suppose $\cC=(I^{\cC}, \{p^{\cC}_{\vec{i}},\epsilon^{\cC}_{\vec{i}}\}_{\vec{i}\in I^{\cC}})$ is a consistent configuration  
with $\vec{\infty}\in I^\cC$, and that 
$\{p^{\cC}_{\vec \infty}(j,\infty)\mid j\le d\}$ has no age-changes over $X^+$.
Then there is a computable (in $T$) aged-isomorphic subtree $S$ of $T$ extending $A$ in which the extensions of $D$ extend $p^\cC_{\vec{\infty}}$,
  and in which all extensions of $D$ to a copy of $\Delta$ have color $\epsilon^\cC_{\vec{\infty}}$.
\end{lemma}

\begin{proof}
We inductively construct  such an $S$.
First we must ensure that all members of $\mathcal{Y}$ that are in the tree $S$ we are building will extend the set $\{p^\cC_{\vec \infty}(j,\infty)\mid j\le d\}$.
Toward that end, 
let $A^+$ denote the immediate successors of the  maximal nodes in $A$.
Let $Y_0$ denote the set 
$\{p^\cC_{\vec \infty}(j,\infty)\mid j\le d\}$, and 
let $Y_1$ denote the set of leftmost extensions of the maximal nodes in $A^+\setminus X^+$ with the same length as the nodes in $Y_0$.
Let $Z=Y_0\cup Y_1$.
We point out that if  $\{p^\cC_{\vec \infty}(j,\infty)\mid j\le d\}=X^+$,  then $Z$ is exactly $A^+$.
Otherwise, $Z$ is a proper end-extension of $A^+$ with no age-change over $A^+$.

Let $T(n)$ denote the set of those nodes in the $n$-th level of $T$, and (using Ramsey space notation) let 
$r_n(T)$ denote $\bigcup_{m<n}T(m)$,  the
{\em $n$-th initial segment} of $T$.
Let $N\sse \om$ be the set of those  $n$ such that $r_{n+1}(T)$ contains at least one copy $\Delta'$ of $\Delta$  extending $D$, with the maximal nodes of $\Delta'$ being maximal in $T(n)$.
Equivalently, $N$ is the set of those $n\in\om$ where $T(n)$ contains at least one member of $\mathcal{Y}$.
Let $n_0=\min(N)$, and let $k$ be such that $A=r_k(T)$.
Note then that $X^+$ is a subset of $T(k)$.
If $k<n_0$, let $r_{n_0}(S)$ be an extension of $A$ such that  $S(k)$  is an end-extension of $Z$ with no age-changes over $Z$.
If $k=n_0$, then  $r_{n_0}(S)=A$.

Suppose $n\in N$ and  we have constructed  $r_n(S)$. 
We will build an instantiation of $\cC$ that will be used to build $S(n)$. Let $U$ denote the set of 
all  immediate successors (in $T$) of the nodes in   $S(n-1)$.
 Let $V$ be a level set extension of $U$ where $r_n(S) \cup V$ is aged-isomorphic to $r_{n+1}(T)$.
This implies
that
$V$ has no age-changes over $U$.
For each $j\le d$, let 
$m_j=|V_j|$ and let 
$V_j=\{t_{j,i}\mid i<m_j\}=V\cap T_j$.
Fix some $b>\max\{i<\om:\exists j<d\, ((j,i)\in\dom(p^\cC_{\vec \infty}))\}$, and 
let $q$ be a condition with 
\begin{enumerate}
    \item[$\bullet$]
$\dom(q)=\{(j,i)\in\dom(p^\cC_{\vec\infty}):i\ne\infty\}\cup
\bigcup_{j<d}\{j\}\times [b,b+m_j-1]$,
    \item[$\bullet$]
$q(j,i)\contains p^\cC_{\vec\infty}(j,i)$ for each $(j,i)\in \dom(p^\cC_{\vec\infty})$,
 and 
 \item[$\bullet$]
 $q(j,i)=t_{j,i-b}$ for $j\le d$ and $i\in [b,b+m_j-1]$.
 \end{enumerate} 
Define $\rho$ on $\dom(q)$ by 
setting $\rho(j,i)=(j,i)$ for each $(j,i)\in \dom(p^\cC_{\vec\infty})$ and 
$\rho(j,i)=(j,\infty)$ for $j\le d$ and  $i\in [b,b+m_j-1]$.

Let $F$ be the set  
consisting of all tuples $\vec{i}=(i_0,\dots,i_d)\in \prod_{j\le d} [b,b+m_j-1]$
such that the set of nodes $\{t_{j,i_j-b}\mid j\le d\}  $
has no age-changes over the set of nodes $\{p^\cC_{\vec \infty}(j,\infty)\mid j\le d\}$.
Then $(F,q,\rho)$ is an instantiation of $\cC$, since $q\le (p^\cC_{\vec\infty})^{\infty\mapsto \vec i}$ for each $\vec i\in F$.
Since $I^\cC$ is consistent, there is an extension $r\le q$ such that for each $\vec i\in F$, $\chi(r(\vec i))=\epsilon^\cC_{\vec\infty}$.
Let $W=\{r(j,b+i)\mid j\le d,\ i<m_j\}$.
Let $W'$ be the set of the leftmost extensions of the nodes in 
 $V\setminus \bigcup_{j\le d} V_j$ to the length of the nodes in $W$, and   let $W^*=W\cup W'$.
 Define $r_{n+1}(S)=r_{n}(S)\cup W^*$.
Then  each  member of $\mathcal{Y}$ in $S(n)$ has color $\epsilon^\cC_{\vec\infty}$.

Given that $r_{n+1}(S)$ has been constructed, to carry on the induction, let $n'$ be the least member in $N$ above $n$ and let $r_{n'}(S)$ be any extension of $r_{n+1}(S)$  that is aged-isomorphic to $r_{n'}(T)$.
Then repeat the above procedure to build $S(n')$.
To finish, we let $S=\bigcup_{n\in N}r_{n+1}(S)$.
Then by induction, for each $n\in N$, every member of $\mathcal{Y}$ in $S(n)$ has color $\epsilon^\cC_{\vec\infty}$.
\end{proof}

It remains now to show that there is a consistent configuration whose domain contains $\vec\infty$.  To do this, it suffices to find consistent configurations with any $\vec i^*$ in the domain, proceeding inductively on the number of times $\infty$ appears in $\vec i^*$.
The base case, where $\vec i^*\in \om^{d+1}$, is essentially the argument we gave for simple configurations with some additional bookkeeping.

\begin{lemma}[$\mathsf{ACA}_0$]\label{lem8}
If $\cC$ is a consistent configuration, $\pi\in(\omega\rightarrow\omega)^{d+1}$ is a tuple of injective functions, and $\vec i^*\in\om^{d+1}$ then there is a consistent configuration $\cD\le_\pi \cC$ with $\vec i^*\in I^\cD$.
\end{lemma}

\begin{proof}
Again, let $\cC_0$ be the condition $\leq_\pi\cC$ with $I^{\cC}=\pi[I^{\cC}]$, and assume $\vec i^*\not\in I^\cC$.
Let $\cC'$ be the configuration extending $\cC_0$ by setting $I^{\cC'}=I^{\cC_0}\cup\{\vec i^*\}$, letting
$p^{\cC'}_{\vec i^*}$ be the condition with 
$\dom(p^{\cC'}_{\vec i^*})=\{(j,i^*_j)\mid j\le d\}$ and $p^{\cC'}_{\vec i^*}(j,i^*_j)=
x_j$ for all $j\le d$, and letting 
 $\epsilon^{\cC'}_{\vec i^*}=0$.
If $\cC'$ is consistent, we are done, so suppose that $\cC'$ is not consistent.  Then there is an instantiation  $(K^*, q^*,\rho^*)$
of $\mathcal{C}'$ satisfying the following:
For all but finitely many $\ell<\om$, for all $r\le q^*$ with $\ell_r=\ell$, there is an $\vec{i}\in K^*$ such that $\chi(r(\vec{i}))\ne \epsilon^{\cC}_{\rho^*(\vec{i})}$.
 Since $\cC$ is consistent, we must have $\vec{i}^*\in K^*$ 
 and   
$p^{\cC'}_{\rho^*(\vec{i}^*)}=p^{\cC'}_{\vec{i}^*}$, $\rho^*(j,i^*_j)=(j,i^*_j)$ for all $j\le d$, and also $q^*\le p^{\cC'}_{\rho^*(\vec{i}^*)}$.
We partition $\dom(q^*)$ as $E^*_0\cup E^*_\infty$, where $E^*_0=\{(j,a)\mid \rho^*(j,a)=(j,a)\}$ and $E^*_{\infty}=\{(j,a)\mid \rho^*(j,a)=(j,\infty)\}$.
Let $q^*_0=q^*\re E^*_0$.

Let $\cD$ be the configuration extending $\cC_0$ with
$I^{\cD}=I^{\cC}\cup\{\vec{i}^*\}$, 
$p^{\cD}_{\vec{i}^*}=q^*_0$, and $\epsilon^{\cD}_{\vec{i}^*}=1$.
We claim that $\cD$ is consistent. 
Consider any 
instantiation $(K,q,\rho)$ of $\mathcal{D}$.
If  $\vec{i}^*\not\in K$, then the claim follows from the consistency of $\cC$, so we may assume that
$\vec{i}^*\in K$.

As before, we partition $\dom(q)=E_0\cup E_\infty$ where $E_0=\{(j,a) \mid \rho(j,a)=(j,a)\}$ and $E_\infty=\{(j,a) \mid \rho(j,a)=(j,\infty)\}$. 
Note that $E_0\contains E_0^*$, since $q\le q^*_0$ as $\vec i^*\in K$.
Without loss of generality, we will assume $E_\infty\cap E^*_\infty=\emptyset$; this is without loss of generality because when $(j,a)\in E_\infty$, we can replace $(j,a)\in\dom(q)$ with some $(j,a')$ not in $E^*_\infty$, adjusting $K$ and $\rho$ accordingly.
We then construct $q'$ so that $q'\leq q$ and $q'\le q^*$ as follows:
Set $\dom(q')=\dom(q)\cup E^*_\infty$.  For each $(j,a)\in \dom(q)$, let $q'(j,a)=q(j,a)$.
For $(j,a)\in E^*_\infty$, set $q'(j,a)$ to be the leftmost extension $q^*(j,a)$ of length $\ell_q$, with one possible exception:
If we are in 
 Case 2 where there is a coding node in the maximum level of $\Delta$, and 
$(d,a) \in E^*_\infty$, 
then  $q^*(d,a)$ is the coding node of length $\ell_{q^*}$. In this case we must define $q'(d,a)$ to be the coding node in $\ran(q)$.
In all cases, $\ran(q')\re \dom(q^*)$ 
has no age changes over  
$\ran(q^*)$.
 Thus, $q'\le q^*$;   it is easily seen that $q'\le q$ since $q'\re \dom(q)=q$.

Now  take $K'=K\cup K^*$ and $\rho'=\rho\cup \rho^*$.
Observe that $(K'\setminus\{\vec{i}^*\}, q', \rho'\upharpoonright (K'\setminus\{\vec{i}^*\}))$ is an instantiation of $\cC_0$, and therefore there are infinitely many levels $\ell$ so that there is an $r\leq q'$ with $\ell_r=\ell$ and so that for each $\vec{i}\in K'\setminus\{\vec{i}^*\}$, $\chi(r(\vec{i}))=\epsilon^{\cC_0}_{\vec{i}}$. 
However, since $q'\le q^*$,  it must be that at for all but finitely many of these $\ell$, we have $\chi(r(\vec{i}))=1$, showing that $\cD$ is consistent.
\end{proof}

\begin{lemma}[$\mathsf{ACA}_0$]\label{lem9}
If $\cC$ is a consistent configuration, $\pi\in(\omega\rightarrow\omega)^{d+1}$ is a tuple of injective functions, and $\vec i\not\in\cC$, then there is a consistent configuration $\cD\le_\pi \cC$ with $\vec i\in I^{\cD}$.
\end{lemma}

\begin{proof}
We proceed by induction on the number of indices of $\vec i$ which are $\infty$.
If that number is $0$, then this is Lemma \ref{lem8}. Without loss of generality, we may assume the first index of $\vec i$ is $\infty$, so $\vec i = (\infty,\vec{i}^-)$. 
We will use an additional map $\pi'$: $\pi'_0(a)=2a+1$, and $\pi'_i$ is the identity for $0<i$.

Order the conditions in order-type $\om$ as  $p_0,p_1,\dots$.
We now define a sequence of successive extensions $(\pi'\circ\pi)(\cC)$:
Let $\cC_0=(\pi'\circ\pi)(\cC)$; given 
$\cC_i$, let $\cC_{i+1}$ be the extension with $I^{\cC_{i+1}}=I^{\cC_i}\cup\{(2i+2,\vec{i}^-)\}$ and the pair $p^{\cC_{i+1}}_{(2i+2,\vec{i}^-)}, \epsilon^{\cC_{i=1}}_{(2i+2,\vec{i}^-)}$ chosen least such that $\cC_{i+1}$ is consistent.
The inductive hypothesis guarantees that there is some choice of $p^{\cC_{i+1}}_{(2i+2,\vec{i}^-)}$.

Let $\cC_\om=\bigcup_{i<\om} \cC_i$.
Since each $\cC_i$ is consistent, so is $\cC_\om$.
So there is an extension $\cC_{\om+1}$ with $(0,\vec{i}^-)\in I^{\cC_{\om+1}}$.
Consider $p^*=p^{\cC_{\om+1}}_{(0,\vec{i}^-)}$.
Since each $p^{\cC_{i+1}}_{(2i+2,\vec{i}^-)}$ was chosen least in the list so that the result was consistent, we must have 
$p^{\cC_{i+1}}_{(2i+2,\vec{i}^-)}$ appearing no later than $p^{\cC_{\om+1}}_{(0,\vec{i}^-)}$
in our list of conditions, for every $i$.
In particular, there must be cofinitely many stages  at which 
$p^{\cC_{i+1}}_{(2i+2,\vec{i}^-)}=p^{\cC_{\om+1}}_{(0,\vec{i}^-)}$ 
and $\epsilon^{\cC_{i+1}}_{(2i+2,\vec{i}^-)} = \epsilon^{\cC_{\om+1}}_{(0,\vec{i}^-)}$.

    Therefore, the configuration $\cD$ with $I^{\cD}= I^{\pi(\cC)}\cup\{\vec i\}$, $p^{\cD}_{\vec i}=p^*$, and $\epsilon^{\cD}_{\vec i}=\epsilon^{\cC_{\om+1}}_{(0,\vec{i}^-)}$ must be consistent. 
\end{proof}

To complete the proof of Theorem \ref{thm.general}, observe that we may apply Lemma \ref{lem9} finitely many times to show the existence of a consistent configuration containing $\vec \infty$. Being consistent is a $\Pi_2$ property, so we can search for one using $0^{(2)}$, and then, having found this configuration, apply Lemma \ref{lem.Dendhomog} to build the homogeneous copy $S$. \hfill $\square$

\subsection{Dovetailing}

To construct trees which are simultaneously homogeneous for finitely many diaries, we can dovetail the construction from Theorem  \ref{thm.general}.
Recall that  we allow 
 $\Delta_i$ to be the minimal node of a diary; in this case $\Delta^-_i$  is the empty diary and 
$\Delta_i$ is a splitting node.

\begin{lemma}[$\mathsf{ACA}_0$]\label{lem.3.12}
Let $\Delta_0,\dots,\Delta_{k-1}$ be non-empty initial segments of some diaries, and  let 
$\Delta^-_i$ be the immediate predecessor of $\Delta_i$, $i<k$.
Let $T$ be given and for each $i<k$, let $\chi_i$ be a coloring of all copies of $\Delta_i$ in $T$ into $m_i$ colors. Then there is a computable (in $T$) aged-isomorphic subtree $S$ and colorings $\chi^-_i$ of all copies of $\Delta^-_i$ in $S$ so that whenever $D$ is a copy of $\Delta_i$ in $S$ and $D^-_i$ is the initial copy of $\Delta^-_i$, we have $\chi_i(D)=\chi^-_i(D^-)$.
Further, $S$ is $0''$ in $T$.\end{lemma}

\begin{proof}
We construct $S$ level by level. At each stage, we are committed to working in some computable subcopy $T_n$ of $T$; initially we set $T_0=T$ and take the initial level $A_0$ of $S$ to be the root of $T$.

Suppose we have constructed $A_n$ inside $T_n$. Suppose there is a copy $D$ of some $\Delta^-_i$ whose maximum level is contained in $A_n$. We consider the coloring $\chi'_i$ where $\chi'_i(D)=0$ if $\chi_i(D)<m/2$ and $\chi'_i(D)=1$ if $\chi_i(D)\geq m/2$. Then we apply Lemma \ref{lem9} to find a condition $p^{\cC}_{\vec\infty}(j,\infty)$, and Lemma \ref{lem.Dendhomog} then gives us a subtree $T_{n+1}$ of $T_n$ so that all copies of $\Delta$ extending $D$ have the same $\chi'_i$ color in $T_{n+1}$.

We apply this repeatedly, reducing the number of colors, and considering each copy of any $\Delta^-_i$ whose maximum level is contained in $A_n$, successively choosing subtrees in which the extensions use fewer colors. This requires only finitely many steps, and at the end we let $T_{n+1}$ be the final subtree. We then take $A_{n+1}=r_{n+1}(T_{n+1})$. Each individual tree $T_{n+1}$ is computable, and identifying which computable tree we should continue the construction in is computable from $0''$.
\end{proof}

\begin{theorem}[$\mathsf{ACA}_0$]\label{thm.penultimatestep}
   Given a copy of $\mathbb{H}_{n+1}$, a finite list of diaries $\Delta_0,\ldots,\Delta_{k-1}$, and colorings of all copies of each $\Delta_i$ in $\mathbb{H}_{n+1}$ into $m_i$ colors, there is a subcopy $\mathbb{H}'$ of $\mathbb{H}_{n+1}$ in which, for each $i<k$, every copy of $\Delta_i$ has the same color. Furthermore, $\mathbb{H}'$ is $0^{(2\max\{|\Delta_i|\}-1)}$ in the colorings where $|\Delta_i|$ is the number of levels in $\Delta_i$.
\end{theorem}

\begin{proof}
We apply the preceding lemma repeatedly, successively getting copies of $\mathbb{H}_n$ in which the color of each $\Delta_i$ is determined by shorter and shorter initial segments, until we get a copy in which each copy of $\Delta_i$ has the same color.
The last step of the proof involves colorings of the splitting nodes on the spine.  This is an infinite pigeonhole principle and only requires $0'$.
Hence, $\mathbb{H}'$ is  $0^{(2\max\{|\Delta_i|\}-1)}$ in the colorings rather than $0^{(2\max\{|\Delta_i|\})}$.
\end{proof}

Theorem \ref{thm.upperbound} follows immediately from 
applying Theorem \ref{thm.penultimatestep} to the finitely many diaries representing $G$.

\section{Remarks and future directions}

In this paper, we focused on   Henson graphs for the sake of clarity.
But it should be pointed out that 
the methods in Section 3 can be easily adapted to any infinite structure for which coding trees and forcing methods have been used to prove finite big Ramsey degrees. 
In particular, our methods will yield finite big Ramsey degrees in ACA$_0$ for
free amalgamation classes with finitely many binary relations 
\cite{MR4457343}, for binary relational structures satisfying the SDAP$^+$ \cite{CDP1, CDP2},
and for finite chains in the binary branching pseudotree \cite{CDW}.
Further, the  methods in Section 3 can be applied on diaries for copies of infinite structures such as in   \cite{Dobrinen_SDAP},
\cite{DobrinenRado19},
 and \cite{Dobrinen/Zucker23}. 
 

As mentioned in the Introduction,
there is a gap between
the lower bound of $0^{(|G|-1)}$
in Theorem 39 of  \cite{CDM} and
our upper bound of 
$0^{(2\delta(G,n)-1)}$ in Theorem
\ref{thm.upperbound}.
For example,
in  $\mathbb{H}_3$ there are two diaries for  edges, $K_2$,
and five diaries for disconnected pairs, $\overline{K}_2$; 
each of these diaries  has four levels. 
Thus, given a computable coloring
of the  copies of $K_2$ ($\overline{K}_2$) into finitely many colors,
the complexity of a copy   of $\mathbb{H}_3$ in which each diary for $K_2$  ($\overline{K}_2$) has one color lies somewhere between $0'$ and $0^{(7)}$.
For $\overline{K}_2$ we can actually do slightly better, as Theorem 22 in \cite{CDM} yields $0''$ as a lower bound.
In fact, Theorem 22 in \cite{CDM} yields a lower bound of $0^{(2|G|-2)}$ for anticliques $G$ in any Henson graph.  It will take significant work, likely in terms of improving both the lower (as done in \cite{CDM}) and upper bounds (as done here),  to provide a precise bound.

We mention that Hubi\v{c}ka found a purely combinatorial proof that $\mathbb{H}_3$ has  finite big Ramsey degrees in \cite{Hubicka_CS20} by using
 the Carlson--Simpson Lemma. This proof does not provide tight upper bounds for the big Ramsey degrees. But this proof does show for the generic partial order with extensions has finite big  Ramsey degrees; something it seems that cannot shown via coding trees and forcing methods.  
 

 Liu and Patey recently proved that the Carlson--Simpson Lemma holds in ACA$_0$ in \cite{LP}.   By combining \cite{Hubicka_CS20} and \cite{LP} we have another proof that $\mathbb{H}_3$ has  finite big Ramsey degrees is provable in ACA$_0'$. At this point no one has calculated the number of jumps needed to bound the complexity of $\mathbb{H}'$, as in Theorem~\ref{thm.penultimatestep}, from the proof using the Carlson--Simpson Lemma.  We suspect they are higher than the ones we have calculated.  But we will leave that as a question. In terms of indivisibility, Liu and Patey were able to answer Questions~47 and 48  but not Question~49 of \cite{CDM} with this technique.

\bibliographystyle{amsplain}
\bibliography{CDT}

\end{document}